\documentclass[12pt]{article}
\usepackage{amssymb}
\usepackage{amsmath}
\usepackage{amsfonts}
\usepackage{graphicx}
\usepackage{amsthm}
\usepackage{color}
\usepackage{tikz}
\usetikzlibrary{matrix}

\usepackage{xy}
\input xy

\newcommand{\bea}{\begin{eqnarray}}
\newcommand{\eea}{\end{eqnarray}}

\newcommand{\be}{\begin{equation}}
\newcommand{\ee}{\end{equation}}

\theoremstyle{plain}
\newtheorem{theorem}{Theorem}
\newtheorem{corollary}[theorem]{Corollary}
\newtheorem{proposition}[theorem]{Proposition}
\newtheorem{lemma}[theorem]{Lemma}
\newtheorem*{theorem*}{Theorem}

\theoremstyle{definition}
\newtheorem{definition}[theorem]{Definition}

\newtheorem{example}{Example}

\newcommand{\beas}{\begin{eqnarray*}}
	\newcommand{\eeas}{\end{eqnarray*}}

\numberwithin{equation}{section}
\numberwithin{theorem}{section}

\usepackage[colorlinks=true]{hyperref}
\hypersetup{urlcolor=blue, citecolor=red, linkcolor=blue}

\usepackage{colortbl}

\usepackage{graphicx}
\usepackage{amsmath,amssymb}
\usepackage{bm}
\usepackage{enumerate}
\usepackage{tikz-cd}
\usetikzlibrary{cd}

\renewcommand{\d}{\mathrm{d}}
\newcommand{\T}{\mathrm{T}}
\newcommand{\Ad}{{\rm Ad}}
\newcommand{\ad}{{\rm ad}}
\newcommand{\R}{{\mathbb{R}}}

\def\email#1{{\tt#1}}

\begin{document}
	
	\centerline{\Large {\bf Reduction of exact symplectic manifolds}}
    \centerline{\Large{\bf and energy hypersurfaces}}
	\vskip 0.5cm
	
	\centerline{ J. Lange$^{1}$ and B.M. Zawora$^{2}$
	}
	\vskip 0.5cm
	
	\centerline{Department of Mathematical Methods in Physics, University of Warsaw,}
	\medskip
	\centerline{ul. Pasteura 5, 02-093 Warszawa, Poland}
    \medskip \centerline{\email{j.lange2@uw.edu.pl$^1$, b.zawora@uw.edu.pl$^2$}}
	
	\vskip 1cm
	
	\begin{abstract}{
This article introduces two reduction schemes for Hamiltonian systems on an exact symplectic manifold admitting Lie group symmetries. It is demonstrated that these reduction procedures are equivalent, by employing a modified Marsden–Meyer–Weinstein reduction theorem for exact symplectic manifolds and contact manifolds given by energy hypersurfaces. Each approach is illustrated through an example.}

		
		
        
		
	\end{abstract}

{\setcounter{tocdepth}{2}
\def\baselinestretch{1}
\small
\def\addvspace#1{\vskip 1pt}
\parskip 0pt plus 0.1mm
\tableofcontents
}

	\section{Symplectic geometry}
 A {\it symplectic manifold} is a pair $(P,\omega)$, where $\omega\in \Omega^2(P)$ is closed and non-degenerate. If $\omega=-\d\theta$, for some $\theta\in \Omega^1(P)$, then $(P,\theta)$ is called an {\it exact symplectic manifold} \cite{AM_78}. The {\it Liouville vector field} associated with \((P,\theta)\) is the unique vector field $\nabla$ on $P$ such that $\theta$ is the contraction of $\omega$ with $\nabla$, i.e. $\theta=\iota_\nabla\omega$. Then, the triple $(P,\theta,\nabla)$ is an exact symplectic manifold $(P,\theta)$ with the Liouville vector field $\nabla$. Moreover, $\langle\cdot,\cdot\rangle$ denotes the natural pairing between a Lie algebra $\mathfrak{g}$ and its dual space $\mathfrak{g}^*$. Let $\mathfrak{X}(P)$ be the Lie algebra of vector fields on $P$. A vector field $X\in \mathfrak{X}(P)$ is {\it Hamiltonian} if $\iota_{X}\omega=df$ for some $f\in C^{\infty}(P)$. Then, $f$ is called a {\it  Hamiltonian function} of $X$. Since $\omega$ is non-degenerate, every   $f\in C^\infty(P)$ is the Hamiltonian function of a unique Hamiltonian vector field $X_f$ (see \cite{AM_78,OR_04}). Later on, it is assumed that all structures are smooth and defined over the reals. The manifolds are Hausdorff, connected, paracompact, and finite-dimensional.

\begin{definition}
\label{Def::Onohomoaction}
   A Lie group action $\Phi\colon G\times P\to P$ is an \textit{exact symplectic Lie group action} relative to $(P, \theta)$ if $\Phi_g^*\theta=\theta$ for each $g\in G$, where $\Phi_g\colon P\ni p \mapsto \Phi_g(p):=\Phi(g,p)\in P$. In other words, $    \mathcal{L}_{\xi_P}\theta=0$ for any $\xi\in\mathfrak{g}$.
\end{definition} 
For an exact symplectic Lie group action $\Phi\colon  G\times P\rightarrow P$, it follows that $\Phi_{g*}\nabla=\nabla$.

A symplectic momentum map for an exact symplectic manifold \((P,\theta)\) is defined using the fact that \(\omega=-\d\theta\).

\begin{definition}
\label{Def::MomentumMap1homo}
    An exact symplectic momentum map associated with $(P,\theta)$ and a Lie group action $\Phi\colon G\times P\rightarrow P$ is a map $\mathbf{J}_\theta^\Phi\colon P\rightarrow \mathfrak{g}^*$ such that
    \[
    \iota_{\xi_P}\theta=\langle \mathbf{J}^\Phi_\theta,\xi\rangle,\qquad \forall \xi\in\mathfrak{g}\,,
    \]
    where $\xi_P$ is the fundamental vector field of $\Phi$ related to $\xi\in\mathfrak{g}$, namely \linebreak $\xi_P(p) = \frac{\d}{\d t}\bigg|_{t=0}\Phi(\exp(t\xi),p)$ for any $p\in P$ and $\exp\colon \mathfrak{g}\rightarrow G$ stands for the exponential map.
\end{definition}

An exact symplectic momentum map $\mathbf{J}^\Phi_\theta\colon P\rightarrow \mathfrak{g}^*$ associated with a Lie group action $\Phi\colon G\times P\rightarrow P$ is {\it \(\Ad^{*}\)-equivariant}, i.e. \(\Ad^*_{g^{-1}}\circ \mathbf{J}^\Phi_\theta=\mathbf{J}_\theta^\Phi\circ\Phi_g\) for any \(g\in G\) \cite{AM_78,LRVZ_np}.

Let us recall some technical notions. A {\it  regular value} \cite{AM_78,Lee_13} of a map $F\colon M\rightarrow N$ is a point $y\in N$ such that for each $x\in F^{-1}(y)$ the map $\T_xF\colon \T_x M\rightarrow \T_{F(x)}N $ is surjective.

Now, consider the pre-image of the orbit of $\mu\in\mathfrak{g}^*$ relative to the natural action \(\mathbb{R}^\times\times\mathfrak{g}^*\ni(\lambda,\mu)\longmapsto \lambda\mu\in \mathfrak{g}^*\), where $\mathbb{R}^\times:= \mathbb{R}\setminus\{0\}$, namely
\[
(\mathbf{J}^\Phi_\theta)^{-1}(\mathbb{R}^{\times }\mu)=\{ p\in P\,\mid\, \exists\, \lambda\in\mathbb{R}^\times,\,\, \mathbf{J}^\Phi_\theta(p)=\lambda\mu\}.
\]
If $\mu$ is a regular value of $\mathbf{J}^\Phi_\theta$, the spaces $(\mathbf{J}^\Phi_\theta)^{-1}(\mu)$ and $(\mathbf{J}^\Phi_\theta)^{-1}(\mathbb{R}^\times\mu)$ are submanifolds of $P$ \cite{LRVZ_np}. Additionally, a Lie group action $\Phi\colon  G\times M\rightarrow M$ is {\it quotientable} when its space of orbits, denoted by $M/G$, is a manifold and the projection $\pi\colon M\rightarrow M/G$ is a submersion. If $\Phi$ is free and proper, it is quotientable \cite{AM_78,Lee_13}.

The following proposition is crucial in the Marsden-Meyer-Weinstein (MMW) reduction; the proof can be found in \cite{LRVZ_np}.

\begin{proposition}
\label{Prop::Kgroup}
    Let $\mathfrak{k}_{[\mu]}:=\ker\mu\cap\mathfrak{g}_{[\mu]}$, where $\ker\mu=\{\xi\in\mathfrak{g}\,\mid\,\langle\mu,\xi\rangle =0\}$ and $\mathfrak{g}_{[\mu]}=\{\xi\in \mathfrak{g}\,\mid\,\ad^{*}_{\xi}\mu\wedge \mu =0\}$. Then, $\mathfrak{k}_{[\mu]}$ is a Lie subalgebra of $\mathfrak{g}$.
\end{proposition}

Proposition \ref{Prop::Kgroup} implies that there exists a unique connected and simply connected Lie subgroup of $G$, denoted as $K_{[\mu]}$, whose Lie algebra is $\mathfrak{k}_{[\mu]}$.

To simplify the notation, we call $(P,\theta,\nabla,h)$ an {\it exact symplectic Hamiltonian system}, where $(P,\theta,\nabla)$ is an exact symplectic manifold with the Liouville vector field $\nabla$ and $h\in C^\infty(P)$. Furthermore, if $h$ is $G$-invariant, then $(P,\theta,\nabla,h)$ is a {\it $G$-invariant exact Hamiltonian system}.

Now, the following proposition states the modified MMW reduction proven in \cite{LRVZ_np}.

\begin{theorem}
\label{Th::OnehomosymRed}
    Let $(P,\theta,\nabla,h)$ be a $G$-invariant exact symplectic Hamiltonian system, let $\mu\in\mathfrak{g}^*$ be a regular value of an exact symplectic momentum map $\mathbf{J}^\Phi_\theta\colon P\rightarrow \mathfrak{g}^*$ \linebreak associated with an exact symplectic Lie group action $\Phi\colon G\times P\rightarrow P$ that is quotientable on $(\mathbf{J}^\Phi_\theta)^{-1}(\mathbb{R}^{\times}\mu)$ by $K_{[\mu]}$. Then, one obtains an exact symplectic Hamiltonian system $(P_{[\mu]}\!:=\!(\mathbf{J}^\Phi_\theta)^{-1}(\mathbb{R}^{\times }\mu)/K_{[\mu]},\theta_{[\mu]},\!\!\nabla_{[\mu]},\!h_{[\mu]})$ such that
    \[
    \tau^*_{[\mu]}\omega_{[\mu]}=\jmath_{[\mu]}^*\omega\,,\qquad \tau_{[\mu]*}\nabla=\nabla_{[\mu]}\,,\qquad \tau_{[\mu]}^*h_{[\mu]}=\jmath_{[\mu]}^*h\,,
    \]
    where $\tau_{[\mu]}\colon(\mathbf{J}_\theta^\Phi)^{-1}(\mathbb{R}^{\times}\mu)\rightarrow P_{[\mu]}$ and $\jmath_{[\mu]}\colon(\mathbf{J}^\Phi_\theta)^{-1}(\mathbb{R}^{\times}\mu)\hookrightarrow P$ are the canonical projection and canonical immersion, respectively.
\end{theorem}

Note that in Theorem \ref{Th::OnehomosymRed} the Liouville vector field $\nabla$ denotes both a vector on $(\mathbf{J}^\Phi_\theta)^{-1}(\R^\times\mu)$ and on $P$, since $\nabla$ is tangent to $(\mathbf{J}^\Phi_\theta)^{-1}(\R^\times \mu)$, see \cite{LRVZ_np}.

The following proposition ensures the reduction of a Hamiltonian vector field.

\begin{proposition}
\label{Prop::OneHomoDynRed}
    Let assumptions of Theorem \ref{Th::OnehomosymRed} hold. Then, the flow $\mathcal{F}_t$ of $X_h$ leaves $(\mathbf{J}^{\Phi}_{\theta})^{-1}(\R^{\times } \mu)$ invariant and induces a unique flow $\mathcal{K}_t$ on $(\mathbf{J}^{\Phi}_{\theta})^{-1}(\R^{\times} \mu)/K_{[\mu]}$ satisfying \linebreak  $\pi_{[\mu]}\circ \mathcal{F}_t=\mathcal{K}_t\circ \pi_{[\mu]}$.
\end{proposition}

\section{Contact geometry}
A {\it contact manifold} is a pair $(M,\mathcal{C})$, where $M$ is a $(2n-1)$-dimensional manifold and $\mathcal{C}$ is the so-called {\it contact distribution}, that is, a {\it maximally non-integrable distribution} with corank one on $M$, i.e. locally is given by $\mathcal{C}|_U=\ker\eta$, for some open neighbourhood $U$ on $M$ and $\eta\in \Omega^1(U)$ such that $\eta\wedge (\d\eta)^n$ is a volume form on $U$ and $n$ is a positive natural number \cite{Ge_08}. Then, $\eta$ is called a {\it contact form}. A contact manifold $(M,\mathcal{C})$ is called {\it co-orientable} if it admits a contact form $\eta\in\Omega^1(M)$ and in that case will be denoted as $(M,\eta)$. A co-orientable contact manifold $(M,
\eta)$ admits a unique vector field $R$ satisfying $\iota_R\eta=1$ and $\iota_R\d\eta=0$ called the {\it Reeb vector field} \cite{Ge_08}. A diffeomorphism on $M$ which preserves $\mathcal{C}$ is called a {\it contactomorphism}, additionally, if it leaves $\eta$ invariant, it is called a {\it strong contactomorphism}. 

\begin{definition}
    A Lie group action $\Phi\colon G\times M\to M$ is a \textit{contact Lie group action} relative to $(M,\eta)$ if the map $\Phi_g\colon M\rightarrow M$ is a strong contactomorphism for any $g\in G$. In other words, $\Phi^*_g\eta=\eta$ for any $g\in G$.
\end{definition}

Then, one can define a contact momentum map in the following manner.

\begin{definition}
\label{Def::ContactMomentumMap}
     A contact momentum map associated with $(M,\eta)$ and a Lie group action $\Phi\colon G\times M\rightarrow M$ is a map $\mathbf{J}^\Phi_\eta\colon M\rightarrow \mathfrak{g}^*$ such that
    \[
    \iota_{\xi_M}\eta=\langle \mathbf{J}^\Phi_\eta,\xi\rangle,\qquad \forall \xi\in\mathfrak{g}\,,
    \]
    where $\xi_M$ is the fundamental vector field associated with $\Phi$ and $\xi$.
\end{definition}
Similarly, a contact momentum map $\mathbf{J}^\Phi_\eta\colon M\rightarrow \mathfrak{g}^*$ associated with a Lie group action $\Phi\colon G\times M\rightarrow M$ is \(\Ad^{*}\)-equivariant \cite{AM_78,OR_04}.

The following theorem presents the MMW reduction theorem for co-orientable contact manifolds. Remarkably, this theorem can be generalised to all contact manifolds, see \cite{LRVZ_np}. However, this work is primarily concerned only with co-orientable contact manifolds.

\begin{theorem}
\label{Th::conRed}
Let $(M,\eta)$ be a contact manifold, and let $\Phi\colon G\times M\rightarrow M$ be a contact Lie group action that is quotientable on $(\mathbf{J}^\Phi_\eta)^{-1}(\mathbb{R}^\times\mu)$ by $K_{[\mu]}$. Assume that $\mu\in \mathfrak{g}^*$ is a regular value of a contact momentum map $\mathbf{J}^\Phi_\eta\colon M\rightarrow \mathfrak{g}^*$ associated with $\Phi$. Then, $\left(M_{[\mu]}:=(\mathbf{J}^\Phi_\eta)^{-1}(\mathbb{R}^\times\mu)/K_{[\mu]},\eta_{[\mu]}\right)$ is a contact manifold, while $\eta_{[\mu]}$ is uniquely defined by
\[
\pi_{[\mu]}^*\eta_{[\mu]}=i_{[\mu]}^*\eta\,,
\]
where $i_{[\mu]}\colon (\mathbf{J}^\Phi_\eta)^{-1}(\mathbb{R}^{\times} \mu)\hookrightarrow M$ and $\pi_{
[\mu]}\colon(\mathbf{J}^\Phi_\eta)^{-1}(\mathbb{R}^{\times}\mu)\rightarrow (\mathbf{J}^\Phi_\eta)^{-1}(\mathbb{R}^{\times} \mu)/K_{[\mu]}$ are the natural immersion and the canonical projection, respectively. Additionally, the Reeb vector field $R$ associated with $(M,\eta)$ projects onto the Reeb vector field $R_{[\mu]}$ associated with $(M_{[\mu]},\eta_{[\mu]})$, i.e. $\pi_{[\mu]*}R=R_{[\mu]}$.
\end{theorem}
Similarly, in Theorem \ref{Th::conRed} the Reeb vector field $R$ denotes both a vector on $(\mathbf{J}^\Phi_\eta)^{-1}(\R^\times\mu)$ and on $M$, since $R$ is tangent to $(\mathbf{J}^\Phi_\eta)^{-1}(\R^\times\mu)$, see \cite{LRVZ_np}.

\section{Reduction of energy hypersurfaces}
\label{Sec::RestrRed}
This section presents the following scheme of a reduction: first, a modified symplectic MMW reduction procedure and then the restriction to the level set of the reduced Hamiltonian function. An illustrative example is provided.

Recall the crucial notion used hereafter. A vector field $Y\in\mathfrak{X}(M)$ is {\it transverse} to a submanifold $N$ of $M$ if $\langle Y_x\rangle+\T_xN=\T_xM$ at each $x\in N$. Then, we write $Y\pitchfork N$.

The following lemma shows how contact geometry arises from studying Hamiltonian dynamics \cite[Lemma 1.4.5]{Ge_08}.

\begin{lemma}
\label{Lemm::hypersurf}
    Let $(P,\theta,\nabla,h)$ be an exact Hamiltonian system such that $\nabla\pitchfork h^{-1}(c)$ for a regular value $c\in\mathbb{R}$ of $h\colon P\rightarrow \R$. Then, $(S:=h^{-1}(c),\eta:=i^*\theta)$ is a contact manifold, where $i\colon S\hookrightarrow P$ is the natural embedding. Moreover, the Reeb vector field associated with $\eta\in \Omega^1(S)$ is a rescaled Hamiltonian vector field associated with $h$, that is, the flows of $X_h$ and $R$ coincide up to reparametrisation.
\end{lemma}

Contact manifolds $(S,\eta)$ that arise from Lemma \ref{Lemm::hypersurf} are called {\it energy hypersurfaces of $(P,\theta,\nabla,h)$}.

\begin{lemma}
\label{Lemm::SymIndCon}
    Let $(S,\eta)$ be an energy hypersurface of an $G$-invariant exact Hamiltonian system $(P,\theta,\nabla,h)$. Then, an exact symplectic momentum map $\mathbf{J}^\Phi_\theta\colon P\rightarrow \mathfrak{g}^*$ associated with a Lie group action $\Phi\colon G\times P\rightarrow P$ induces a contact momentum map $\mathbf{J}^{\Phi^S}_\eta\colon S\rightarrow \mathfrak{g}^*$ related to a contact Lie group action $\Phi^S\colon G\times S\rightarrow S$.
\end{lemma}
\begin{proof}
    As $h\in C^\infty(P)$ is $G$-invariant, $\Phi\colon G\times P\rightarrow P$ induces a Lie group action $\Phi^S\colon G\times S\rightarrow S$ such that $\Phi_g\circ i=i\circ\Phi^S_g$ for any $g\in G$. Moreover, since $\theta$ is $G$-invariant it follows that
     \[
     \Phi^{S*}_g\eta=\Phi^{S*}_gi^*\theta=i^*\Phi^*_g\theta=i^*\theta=\eta,\qquad \forall g\in G.
     \]
     Therefore, $\Phi^S\colon G\times S\rightarrow S$ is a contact Lie group action. Furthermore, $\mathbf{J}^\Phi_\theta$ induces $\mathbf{J}^{\Phi^S}_\eta$ such that $i^*\mathbf{J}^\Phi_\theta=\mathbf{J}^{\Phi^S}_\eta$ and $\langle \mathbf{J}^{\Phi^S}_\eta,\xi\rangle =\iota_{\xi_S}\eta$ for every $\xi\in\mathfrak{g}$, where $\xi_S$ is the fundamental vector field associated with $\Phi^S$ and $\xi\in\mathfrak{g}$.
\end{proof}

Theorem \ref{Th::conRed} implies the following.
\begin{corollary}
\label{cor:1}
Let $(S,\eta)$ be an energy hypersurface of $(P,\theta,\nabla,h)$, let $\mathbf{J}^{\Phi^S}_\eta\colon  S\rightarrow \mathfrak{g}^*$ be a contact momentum map associated with the Lie group action $\Phi^S\colon G\times S\rightarrow S$ induced by an exact symplectic momentum map $\mathbf{J}^{\Phi}_\theta\colon P\rightarrow \mathfrak{g}^*$. Additionally, assume that $\mu\in\mathfrak{g}^*$ is a regular value of $\mathbf{J}^{\Phi^S}_\eta$ and $\Phi_S$ acts in a quotientable manner on $(\mathbf{J}^{\Phi^S}_\eta)^{-1}(\R^\times\mu)$. Then, \(\left(M_{[\mu]}=(\mathbf{J}^{\Phi^S}_\eta)^{-1}(\mathbb{R}^\times\mu)/K_{[\mu]},\eta_{[\mu]}\right)\) is a contact manifold with 
\[
i_{[\mu]}^*\eta=\pi_{[\mu]}^*\eta_{[\mu]}\,,\quad R_{[\mu]}=\pi_{[\mu]*}R\,,
\]
where $\pi_{[\mu]}\colon(\mathbf{J}^{\Phi^S})^{-1}(\mathbb{R}^\times\mu)\rightarrow (\mathbf{J}^{\Phi^S})^{-1}(\mathbb{R}^\times\mu)/K_{[\mu]}$ is the canonical projection, \linebreak $i_{[\mu]}\colon(\mathbf{J}^{\Phi^S}_\mu)^{-1}(\mathbb{R}^\times\mu)\rightarrow S$ is the natural immersion, and $R$ and $R_{[\mu]}$ are Reeb vector field associated with $(S,\eta)$ and $(M_{[\mu]},\eta_{[\mu]})$, respectively.
\end{corollary}

Let us illustrate the above procedure with a simple and illustrative example.
\begin{example}
\label{Ex::1}
Let $Q=\{q\in \mathbb{R}^3 \,\mid\, q_1\neq q_2\}$ and let $P=\T^*\mathbb{R}^3\vert_Q = \bigsqcup_{p\in Q}\T^*_p\mathbb{R}^3$. Consider an exact symplectic manifold $(P,\theta)$, with the canonical Liouville one-form and vector field, respectively, given by $\theta =\!-\!\sum^3_{i=1}p_i \d q_i$ and $\nabla\!=\! \sum^3_{i=1}p_i\frac{\partial}{\partial p_i}$, where $\{q_i,p_j\}_{i,j=1,2,3}$ are adapted coordinates on $P$. Consider $h\colon P\rightarrow \mathbb{R}$ and its associated Hamiltonian vector field $X_h$ of the form
\[
h=\frac{1}{2}\sum^3_{i=1}p_i^2-\frac{1}{\vert q_1-q_2\vert}\,,\quad X_h=\sum^3_{i=1}p_i\frac{\partial}{\partial q_i}-\frac{{\rm sgn}(q_1-q_2)}{(q_1-q_2)^2}\left(\frac{\partial}{\partial p_1}-\frac{\partial}{\partial p_2}\right)\,.
\]
Let $\Phi\colon (\lambda_1,\lambda_2;q,p)\in\mathbb{R}^2\times P\mapsto (q_1+\lambda_1,q_2+\lambda_1,q_3+\lambda_2,p)\in P$ be a Lie group action on $P$. It is free, proper, and leaves $h$ invariant. The fundamental vector fields of $\Phi$ are spanned by $\xi_P^1 = \frac{\partial}{\partial q_1}+\frac{\partial}{\partial q_2}$ and $\xi^2_P=\frac{\partial}{\partial q_3}$.

Note that $\nabla\pitchfork h^{-1}(1/2)$. Therefore, $(S,\eta)$ is an energy hypersurface of $(P,\theta,\nabla,h)$ with 
\[
S:= h^{-1}(1/2)=\{ (\tilde{q}_1,\tilde{q_2},t,\alpha,\phi,\varphi)\in P  \, \mid \, \alpha= 1 \}
\]
and 
\[
\eta = -\sqrt{1+\frac{1}{\vert\beta\vert}}\left(\sqrt{2}\cos\phi\,\d\tilde{q}_1+\sqrt{2}\cos\varphi\sin\phi\,\d\beta+\sin\phi\sin\varphi\,\d t\right),
\]
where $\{\tilde{q}_1,\beta,t,\phi,\varphi\}$ are coordinates on $S$ satisfying 
\begin{gather*}
\tilde{p}_1=\frac{1}{2}(p_1+p_2) = \sqrt{\frac{\alpha+\frac{1}{\vert\tilde{q}_2\vert}}{2}} \cos\phi, \quad  \tilde{p}_2 = \frac{1}{2}(p_1-p_2) = \sqrt{\frac{\alpha+\frac{1}{\vert\tilde{q}_2\vert}}{2}} \sin\phi \cos\varphi,\\ p_3=\tilde{p}_3=\sqrt{\alpha+\frac{1}{\vert\tilde{q}_2\vert}} \sin \phi \sin\varphi\,,\\   \tilde{q}_1=\frac{1}{2}(q_1+q_2), \quad \tilde{q}_2=\frac{1}{2}(q_1-q_2) = \beta , \quad q_3=\tilde{q}_3=t\,.
\end{gather*}
Lemma \ref{Lemm::hypersurf}  yields that the Reeb vector field of $(S,\eta)$ reads
\begin{multline*}
R=\frac{-\cos\phi}{\sqrt{2\left(1+\frac{1}{\vert\beta\vert}\right)}}\frac{\partial}{\partial \tilde{q}_1}-\frac{\sin\phi}{\sqrt{1+\frac{1}{\vert \beta\vert}}}\left(\frac{\cos\varphi}{\sqrt{2}}\frac{\partial}{\partial \beta}+\sin\varphi \frac{\partial}{\partial t}\right)\\ - \frac{\mathrm{sgn}(\beta)}{2 \sqrt{2} \beta^2 \left(1+\frac{1}{\vert\beta\vert}\right)^{3/2}}\left( \frac{\sin\varphi}{\sin\phi}\frac{\partial}{\partial \varphi} -\cos\phi\cos\varphi \frac{\partial}{\partial \phi} \right)
\end{multline*}
Note that $R$ is not well-defined whenever $\sin\phi=0$, since $\tilde{p}$ coordinates are described by spherical-like coordinates which are not global. However, redefining $\phi$ allows one to find suitable coordinates for the points where $\sin\phi=0$ and $R$ will be well-defined. Since the further reduction yields $\sin\phi=1$, these technical calculations will be omitted.

Lemma \ref{Lemm::SymIndCon} gives that a Lie group action $\Phi\colon \R^2\times P\rightarrow P$ induces a Lie group action $\Phi^S\colon \R^2\times S\rightarrow S$ that acts freely and properly on $S$. The fundamental vector fields corresponding to $\Phi^S$ are spanned by $\xi_S^1 = \frac{\partial}{\partial \tilde{q}_1}$ and $\xi_S^2 = \frac{\partial}{\partial t}$. Then, a contact momentum map associated with $\Phi^S$ reads 
\[
\mathbf{J}_\eta^{\Phi^S}\colon (\tilde{q}^1,\beta,t,\varphi,\phi)\in S \longmapsto \left( -\sqrt{2}\,\sqrt{1+\frac{1}{\vert \beta\vert}} \cos\phi, -\sqrt{1+\frac{1}{\vert \beta\vert}}\sin\phi \sin\varphi\right)\in  \R^{2*}\,.
\]
Then, for a regular value $\mu=(0,1)\in\R^{2*}$ of $\mathbf{J}^{\Phi^S}_\eta$, one has 
\[
(\mathbf{J}^{\Phi^S}_\eta)^{-1}(\mathbb{R}^\times \mu) = \{ (\tilde{q}_1, \beta, t, \phi, \varphi) \in S \, \mid \, \cos\phi=0 \wedge \sin \varphi  \neq 0 \}
\]
and $\mathfrak{k}_{[\mu]}=\langle\xi^1\rangle$. Theorem \ref{Th::conRed} yields that $(M_{[\mu]}:= (\mathbf{J}^{\Phi^S}_\eta)^{-1}(\mathbb{R}^\times \mu)/K_{[\mu]},\eta_{[\mu]})$ is a reduced contact manifold with $\eta_{[\mu]}=- \sqrt{1+\frac{1}{\vert \beta\vert}} \left( \sqrt{2}\cos\varphi \, \d \beta +\sin\varphi \, \d t\right)$,
where $\{\beta,t,\varphi\}$ are local coordinates on $M_{[\mu]}$. The reduced Reeb vector field is 
\[
R_{[\mu]}=\frac{-1}{\sqrt{1+\frac{1}{\vert \beta\vert}}}\left( \frac{\cos\varphi}{\sqrt{2}}\frac{\partial}{\partial \beta}+\sin\varphi \frac{\partial}{\partial t}\right) - \frac{\mathrm{sgn}(\beta)}{2 \sqrt{2} \beta^2 \left(1+\frac{1}{\vert\beta\vert}\right)^{3/2}} \sin\varphi\frac{\partial}{\partial \varphi}\,.
\]
\end{example}

\section{Reduced energy hypersurfaces}
\label{Sec::RedRestr}
This section presents the opposite direction of the procedure in Section \ref{Sec::RestrRed}. First, the MMW reduction for exact symplectic manifolds is performed, and then the level set of a reduced Hamiltonian function is considered.

Theorem \ref{Th::OnehomosymRed} establishes that, under given assumptions,  $(P_{[\mu]},\theta_{[\mu]},\nabla_{[\mu]},h_{[\mu]})$ is an exact Hamiltonian system. However, the following lemma is essential to consider the level set of a reduced Hamiltonian function as an energy hypersurface.

\begin{lemma}
\label{Lemm::transvReduced}
    Let $(P,\theta,\nabla,h)$ be an exact $G$-invariant Hamiltonian system and let $c\in \R$ and $\mu\in\mathfrak{g}^*$ be regular values of $h$ and an exact symplectic momentum map $\mathbf{J}^\Phi_\theta\colon P\rightarrow\mathfrak{g}^*$, respectively. If $\nabla\pitchfork h^{-1}(c)$, then it follows that $\nabla_{[\mu]}\pitchfork h_{[\mu]}^{-1}(c)$. 
\end{lemma}
\begin{proof}
Theorem \ref{Th::OnehomosymRed} gives that $\tau^*_{[\mu]}\left(\iota_{\nabla_{[\mu]}}\d h_{[\mu]}\right)=\jmath_{[\mu]}^*\left(\iota_{\nabla}\d h\right)$. As $\jmath_{[\mu]}^*\left(\iota_{\nabla}\d h\right)$ does not vanish, by assumption, $\nabla_{[\mu]}$ is transversal to $h^{-1}_{[\mu]}(c)$.
\end{proof}

Lemma \ref{Lemm::hypersurf} and Lemma \ref{Lemm::transvReduced} applied to \((P_{[\mu]},\theta_{[\mu]},\nabla_{[\mu]},h_{[\mu]})\) imply the following.
\begin{corollary}
        Let $(P_{[\mu]},\theta_{[\mu]},\nabla_{[\mu]},h_{[\mu]})$ be a reduced exact Hamiltonian system. Then, $(S_{[\mu]}:=h_{[\mu]}^{-1}(c), \widetilde{\eta}_{[\mu]}:=\jmath^*\theta_{[\mu]})$ is an energy hypersurface associated with $(P_{[\mu]},\theta_{[\mu]},\nabla_{[\mu]},h_{[\mu]})$, where $\jmath\colon S_{[\mu]}\hookrightarrow P_{[\mu]}$ is the natural embedding.
\end{corollary}

\begin{example}
\label{Ex::2}
Consider the exact symplectic manifold $(P,\theta,\nabla,h)$ in Example \ref{Ex::1} with adapted coordinates $\{q_i,p_j\}_{i,j=1,2,3}$ 
on $P$. Recall that $\Phi\colon \R^2\times P\rightarrow P$ acts freely and properly on $P$ and the corresponding fundamental vector fields are $\xi_P^1=\frac{\partial}{\partial q_1}+\frac{\partial}{\partial q_2}$ and $\xi_P^2=\frac{\partial}{\partial q_3}$. Then, an exact symplectic momentum map associated with $\Phi$ is of the form
\[
\mathbf{J}_\theta^\Phi\colon (q,p)\in P\longmapsto (-p_1 - p_2, -p_3)\in \mathbb{R}^{2*}\,.
\]
As in Example \ref{Ex::1}, let us fix a regular value $\mu = (0,1) \in \mathbb{R}^{2*}$. Then, $(\mathbf{J}^\Phi_\theta)^{-1}(\mathbb{R}^\times \mu) = \linebreak \{ (q,p)\! \in \! P \! \, \! \mid \, p_1+p_2 = 0 \, \wedge \, p_3\neq 0 \}$ and $\mathfrak{k}_{[\mu]}\!=\!\langle \xi^1
\rangle$. Applying Theorem \ref{Th::OnehomosymRed}, $(P_{[\mu]},\theta_{[\mu]},\! \nabla_{[\mu]},h_{[\mu]})$ becomes a four-dimensional exact Hamiltonian system with
\[
\theta_{[\mu]}=-2 \tilde{p}_2 \, \d\tilde{q}_2 - \tilde{p}_3 \, \d \tilde{q}_3,\quad h_{[\mu]} =\frac{1}{2}(2 \tilde{p}_2^2 +\tilde{p}_3^2) - \frac{1}{2 \vert \tilde{q}_2 \vert}\,,\quad\nabla_{[\mu]}=\tilde{p}_2\frac{\partial}{\partial \tilde{p}_2}+\tilde{p}_3\frac{\partial}{\partial \tilde{p}_3}\,,
\]
where $(\tilde{q}_2\!:=\!\frac{1}{2}(q_1\!-\!q_2),\tilde{q}_3\!:=\!q_3,\tilde{p}_2\!:=\!\frac{1}{2}(p_1\!-\!p_2),\tilde{p}_3\!:=\!p_3)$ are new coordinates on $P_{[\mu]}$. By Proposition \ref{Prop::OneHomoDynRed}, the reduced Hamiltonian vector field is of the form 
\[
X_{h_{[\mu]}}=\tilde{p}_2\frac{\partial}{\partial \tilde{q}_2}+\tilde{p}_3\frac{\partial}{\partial \tilde{q}_3}-\frac{{\rm sgn}(\tilde{q}_2)}{4(\tilde{q}_2)^2}\frac{\partial}{\partial \tilde{p}_2}\,.
\]
As in Example \ref{Ex::1}, fix a regular value $c\!=\!1/2$ of $h_{[\mu]}$ and note that $\nabla_{[\mu]}\pitchfork h_{[\mu]}^{-1}(1/2)$. \linebreak Therefore, Lemma \ref{Lemm::SymIndCon} yields that $(S_{[\mu]},\tilde{\eta}_{[\mu]}\!:=\!\jmath^*\theta_{[\mu]})$ is an energy hypersurface of \linebreak $(P_{[\mu]},\theta_{[\mu]},\nabla_{[\mu]},h_{[\mu]})$. Consider the coordinates satisfying
\[
    \tilde{q}_2 = \beta\,,\quad \tilde{q}_3=t \,,\quad \tilde{p}_2 = \sqrt{\frac{\alpha+\frac{1}{\vert\beta\vert}}{2}} \cos \varphi\,,\quad \tilde{p}_3  = \sqrt{\alpha+\frac{1}{\vert\beta\vert}} \sin\varphi\,.
\]
Then, $S_{[\mu]} = \{(\beta, t, a, \varphi) \in   P_{[\mu]} \, \mid \, \alpha=1\}$ and the contact form reads \linebreak $\tilde{\eta}_{[\mu]}=- \sqrt{1+\frac{1}{\vert \beta\vert}} \left( \sqrt{2} \cos\varphi\, \d \beta +\sin\varphi \,\d t\right)$. Furthermore, Lemma \ref{Lemm::SymIndCon} ensures that rescaled $X_{h_{[\mu]}}$, restricted to $S_{[\mu]}$ becomes the Reeb vector field, $\tilde{R}_{[\mu]}$, related to $(S_{[\mu]},\tilde{\eta}_{[\mu]})$, and reads
\[
\tilde{R}_{[\mu]}=\frac{-1}{\sqrt{1+\frac{1}{\vert \beta\vert}}} \left( \frac{\cos\varphi}{\sqrt{2}} \frac{\partial}{\partial \beta} +\sin\varphi \frac{\partial}{\partial t} \right) - \frac{\mathrm{sgn}(\beta)}{2 \sqrt{2} \beta^2 \left(1+\frac{1}{\vert\beta\vert}\right)^{3/2}}\sin \varphi\frac{\partial}{\partial \varphi}\,.
\]
\end{example}

\section{The equivalence of the two procedures}
This section establishes the equivalence between the two procedures introduced in Sections \ref{Sec::RedRestr} and Section \ref{Sec::RestrRed}. Specifically, it is demonstrated that the submanifolds \(S_{[\mu]} = h^{-1}_{[\mu]}(c)\) and \(M_{[\mu]} = (\mathbf{J}^{\Phi^S}_\eta)^{-1}(\mathbb{R}^\times\mu)/K_{[\mu]}\) are diffeomorphic. To illustrate this correspondence, consider the following commutative diagram:

\begin{center}
\small
\begin{tikzcd}
    & P & \\
    (\mathbf{J}^\Phi_\theta)^{-1}(\mathbb{R}^\times\mu)\arrow[ur,hook,"\jmath_{[\mu]}"]\arrow[d,"\tau_{[\mu]}"] & & S\arrow[ul,hook',"i"']\\
    P_{[\mu]}=(\mathbf{J}^\Phi_\theta)^{-1}(\mathbb{R}^\times\mu)/K_{[\mu]}& & (\mathbf{J}^{\Phi^S}_\eta)^{-1}(\mathbb{R}^\times\mu)\arrow[u,hook',"i_{[\mu]}"']\arrow[d,"\pi_{[\mu]}"]\arrow[ull,hook',"i\vert_{(\mathbf{J}^{\Phi^S}_\eta)^{-1}(\R^\times\mu)}"']\\
    S_{[\mu]}\arrow[rr,"\kappa"]\arrow[u,hook',"\jmath"']& & M_{[\mu]}=(\mathbf{J}^{\Phi^S}_\eta)^{-1}(\mathbb{R}^\times\mu)/K_{[\mu]}\arrow[ull,hook',"i_{M_{[\mu]}}"'] 
\end{tikzcd}
\end{center}

If $\mu\in\mathfrak{g}^*$ is a regular value of both $\mathbf{J}^\Phi_\theta$ and $\mathbf{J}^{\Phi^S}_\eta$, Lemma \ref{Lemm::hypersurf} and Lemma \ref{Lemm::SymIndCon} imply the existence of the following embeddings
\[
i\vert_{(\mathbf{J}^{\Phi^S}_\eta)^{-1}(\R^\times\mu)}\colon (\mathbf{J}^{\Phi^S}_\eta)^{-1}(\R^\times\mu) \hookrightarrow (\mathbf{J}^{\Phi}_\theta)^{-1}(\R^\times\mu),\qquad
i_{M_{[\mu]}}\colon M_{[\mu]}\hookrightarrow P_{[\mu]}\,.
\]

\begin{theorem}
    Assume that the hypothesis of Theorem \ref{Th::OnehomosymRed} and Theorem \ref{Th::conRed} hold for an exact Hamiltonian system $(P,\theta,\nabla,h)$ and its associated energy hypersurface $(S,\eta)$, respectively. Let $\kappa\colon S_{[\mu]}\rightarrow M_{[\mu]}$ be defined as $i_{M_{[\mu]}}\circ\kappa=\jmath$. Then, $\kappa$ is a diffeomorphism. 
\end{theorem}
\begin{proof}
    Recall that $i_{M_{[\mu]}}\colon M_{[\mu]}\hookrightarrow P_{[\mu]}$ and $\jmath\colon S_{[\mu]}\hookrightarrow P_{[\mu]}$ are embeddings and $\mu\in\mathfrak{g}^*$ is a regular value of $\mathbf{J}^\Phi_\theta$ and $\mathbf{J}^{\Phi^S}_\eta$. Hence $\kappa$ is bijective. Thus, from the inverse function theorem follows that $\kappa\colon S_{[\mu]}\rightarrow M_{[\mu]}$ is a diffeomorphism.
\end{proof}

Indeed, Example \ref{Ex::1} and Example \ref{Ex::2} show that $S_{[\mu]}\simeq M_{[\mu]}$.

\section*{Acknowledgements}

J. Lange and B.M. Zawora thank the IDUB program mikrogrants from the Faculty of Physics at the University of Warsaw for their partial financial support. B.M. Zawora also acknowledges support from the Spanish Ministry of Science and Innovation grants PID2021-125515NB-C21. The authors also acknowledge Professor J. de Lucas for helpful discussion.


\begin{thebibliography}{8}

\bibitem{AM_78}
R. Abraham, J.E. Marsden.: Foundations of Mechanics. Addison-Wesley Publishing Co., Reading, (1978).

\bibitem{Ge_08}
H. Geiges.: An introduction to contact topology. Cambridge University Press, Cambridge (2008).

\bibitem{Lee_13}
J.M. Lee.: Introduction to smooth manifolds. Springer, New York (2013).

\bibitem{LRVZ_np}
J. de Lucas, X. Rivas, S. Vilariño, B.M. Zawora.: Marsden--Meyer--Weinstein reduction for $k$-contact field theories. Preprint: arxiv:2505.05462.

\bibitem{OR_04}
    J.P. Ortega, T. Ratiu.: Momentum maps and Hamiltonian reduction. Birkh\"{a}user Boston, Inc., Boston (2004).
\end{thebibliography}
\end{document}